\documentclass[reqno,12pt]{amsart}

\usepackage{amsmath, epsfig, cite}
\usepackage{amssymb}
\usepackage{amsfonts}
\usepackage{latexsym}
\usepackage{amsthm}

\newtheorem{theorem}{Theorem}[section]

\newtheorem{corollary}[theorem]{Corollary}
\newtheorem{conjecture}[theorem]{Conjecture}
\newtheorem{lemma}[theorem]{Lemma}

\theoremstyle{remark}

\numberwithin{equation}{section}

\allowdisplaybreaks[1]

\author{Victor J.\ W.\ Guo}
\address{School of Mathematics and Statistics, Huaiyin Normal University,
Huai'an 223300, Jiangsu, People's Republic of China}
\email{jwguo@hytc.edu.cn}
\thanks{The author was partially supported by the National Natural
Science Foundation of China (grant 11771175).}

\title[]
{$q$-Analogues of some supercongruences related to Euler numbers}



\begin{document}

\begin{abstract} Let $E_n$ be the $n$-th Euler number and $(a)_n=a(a+1)\cdots (a+n-1)$ the rising factorial. Let $p>3$ be a prime.
In 2012, Sun proved the that
$$
\sum^{(p-1)/2}_{k=0}(-1)^k(4k+1)\frac{(\frac{1}{2})_k^3}{k!^3} \equiv p(-1)^{(p-1)/2}+p^3E_{p-3} \pmod{p^4},
$$
which is a refinement of a famous supercongruence of Van Hamme.
In 2016, Chen, Xie, and He established the following result:
$$
\sum_{k=0}^{p-1}(-1)^k (3k+1)\frac{(\frac{1}{2})_k^3}{k!^3} 2^{3k} \equiv p(-1)^{(p-1)/2}+p^3E_{p-3} \pmod{p^4},
$$
which was originally conjectured by Sun. In this paper we give $q$-analogues of the above two supercongruences by employing the $q$-WZ method.
As a conclusion, we provide a $q$-analogue of the following supercongruence of Sun:
$$
\sum_{k=0}^{(p-1)/2}\frac{(\frac{1}{2})_k^2}{k!^2}
\equiv (-1)^{(p-1)/2}+p^2 E_{p-3} \pmod{p^3}.
$$
\end{abstract}

\maketitle

\section{Introduction}
In 1914, Ramanujan \cite{Ramanujan} gave a number of rapidly convergent series of $1/\pi$.  Although the following series, due to Bauer \cite{Bauer},
is not listed in \cite{Ramanujan}, it gives an example of this kind:
\begin{align}
\sum^{\infty}_{k=0}(-1)^k(4k+1)\frac{(\frac{1}{2})_k^3}{k!^3} =\frac{2}{\pi}, \label{eq:Ramanujan}
\end{align}
where $(a)_k=a(a+1)\cdots (a+k-1)$ denotes the rising factorial.
Ramanujan's formulas for $1/\pi$ got widely admired in 1980's
when they were discovered to offer efficient algorithms for
calculating decimal digits of $\pi$. See the monograph \cite{BB} of the Borwein brothers.
For a recent proof of Ramanujan's series, see Guillera \cite{Guillera}.

In 1997, Van Hamme \cite{Hamme} developed interesting $p$-adic analogues of Ramanujan-type series. In particular, he conjectured the following supercongruence
corresponding to \eqref{eq:Ramanujan}:
\begin{align}
\sum^{(p-1)/2}_{k=0}(-1)^k(4k+1)\frac{(\frac{1}{2})_k^3}{k!^3} \equiv p(-1)^{(p-1)/2} \pmod{p^3},\label{eq:VAN}
\end{align}
where $p$ is an odd prime. Note that we may calculate the sum in \eqref{eq:VAN} for $k$ up to $p-1$,
since the $p$-adic order of $(\frac{1}{2})_k/k!$ is $1$ for $k$ in the range $(p+1)/2\leqslant k\leqslant p-1$.
Congruences of this kind are called  Ramanujan-type supercongruences. The congruence \eqref{eq:VAN} was first proved by Mortenson \cite{Mortenson} in 2008
using a $_6F_5$ transformation and  the $p$-adic Gamma function, and received a WZ (Wilf--Zeilberger \cite{WZ1,WZ2}) proof  by Zudilin \cite{Zudilin} shortly afterwards.

In 2012, also employing the WZ method, Sun \cite{Sun}
gave the following refinement of \eqref{eq:VAN}: for any prime $p>3$,
\begin{align}
\sum^{m}_{k=0}(-1)^k(4k+1)\frac{(\frac{1}{2})_k^3}{k!^3} \equiv p(-1)^{(p-1)/2}+p^3E_{p-3} \pmod{p^4},\label{eq:Sun}
\end{align}
where $m=p-1$ or $(p-1)/2$, and $E_{p-3}$ is the $(p-3)$-th Euler number.

In recent years, $q$-analogues (or rational function generalizations) of congruences and supercongruences have
aroused the interest of many researchers (see \cite{Guo2018,Guo-t,Guo-a2,Guo-div,Guo-mod4,GPZ,GS3,GS,GZ14,GuoZu,LW,Liu,LP,NP,Straub2,WY0,Zu19}).
For instance, the author \cite{Guo2018} gave the following $q$-analogue of \eqref{eq:VAN}: for any odd integer $n>1$,
\begin{align*}
\sum_{k=0}^{(n-1)/2}(-1)^k q^{k^2} [4k+1]\frac{(q;q^2)_k^3}{(q^2;q^2)_k^3}
\equiv [n]q^{(n-1)^2/4} (-1)^{(n-1)/2}\pmod{[n]\Phi_n(q)^2},  
\end{align*}
Here we need to be familiar with the standard $q$-notation.  The {\it $q$-integer} is defined by $[n]=1+q+\cdots+q^{n-1}$,
and the {\it $q$-shifted factorial} is defined as $(a;q)_n=(1-a)(1-aq)\cdots (1-aq^{n-1})$ for $n\geqslant 1$ and $(a;q)_0=1$.
Moreover, the $n$-th {\it cyclotomic polynomial} $\Phi_n(q)$ is given by
\begin{align*}
\Phi_n(q)=\prod_{\substack{1\leqslant k\leqslant n\\ \gcd(k,n)=1}}(q-\zeta^k),
\end{align*}
where $\zeta$ is an $n$-th primitive root of unity. It is well known that $\Phi_n(q)$ is an irreducible polynomial in $\mathbb{Z}[q]$.

For two rational functions  $A(q)$ and $B(q)$ in $q$ and a polynomial $P(q)$ in $q$ with integer coefficients, we say that $A(q)$ is congruent to $B(q)$
modulo $P(q)$, denoted by $A(q)\equiv B(q)\pmod{P(q)}$,
if $P(q)$ divides the numerator of the reduced form of $A(q)-B(q)$
in the polynomial ring $\mathbb{Z}[q]$.

In this paper, we shall give a $q$-analogue of \eqref{eq:Sun}.
\begin{theorem}\label{thm:1}
Let $n$ be a positive odd integer. Then
\begin{align}
\sum_{k=0}^{N}
(-1)^k q^{k^2} [4k+1]\frac{(q;q^2)_k^3}{(q^2;q^2)_k^3}
&\equiv
 (-1)^{(n-1)/2}q^{(1-n^2)/4}
\left(q^{n\choose 2}[n]+\frac{(n^2-1)(1-q)^2}{24}[n]^3\right) \notag \\
&\quad+[n]^3\sum_{k=1}^{(n-1)/2}\frac{q^k (q^2;q^2)_{k}}{[2k][2k-1](q;q^2)_{k}} \pmod{[n]\Phi_n(q)^3}, \label{eq:main-1}
\end{align}
where $N=(n-1)/2$ or $n-1$.
\end{theorem}

Note that Sun \cite[Equation (3.1)]{Sun2011} proved that, for any odd prime $p$,
$$
\sum_{k=1}^{(p-1)/2}\frac{4^k}{k(2k-1){2k\choose k}}\equiv 2 E_{p-3} \pmod{p}.
$$
Letting $n=p$ be a prime greater than $3$ and taking $q\to 1$ in \eqref{eq:main-1}, we immediately get \eqref{eq:Sun}.

Still using the WZ method, Guillera and Zudilin \cite{GuZu} established the following supercongruence: for odd primes $p$,
\begin{align}
\sum_{k=0}^{(p-1)/2}(-1)^k (3k+1)\frac{(\frac{1}{2})_k^3}{k!^3} 2^{3k} \equiv p(-1)^{(p-1)/2}\pmod{p^3}.  \label{eq:div-3}
\end{align}
Moreover, in 2016, Chen, Xie, and He \cite{CXH} gave the following refinement of \eqref{eq:div-3}:
\begin{align}
\sum_{k=0}^{p-1}(-1)^k (3k+1)\frac{(\frac{1}{2})_k^3}{k!^3} 2^{3k} \equiv p(-1)^{(p-1)/2}+p^3E_{p-3} \pmod{p^4}, \label{eq:div-3-2}
\end{align}
which was originally conjectured by Sun \cite[Conjecture 5.1]{Sun0}. The author \cite{Guo-div} established a $q$-analogue of \eqref{eq:div-3}:
\begin{align}
\sum_{k=0}^{n-1}(-1)^k [3k+1]\frac{(q;q^2)_k^3}{(q;q)_k^3}
\equiv [n]q^{(n-1)^2/4} (-1)^{(n-1)/2} \pmod{[n]\Phi_n(q)^2}. \label{eq:q-Zudilin-3}
\end{align}
We point out that a supercongruence for the left-hand side of \eqref{eq:div-3} modulo $p^4$,
also conjectured by Sun \cite[Conjecture 5.1]{Sun0}, was recently confirmed by Mao \cite{Mao}.

In this paper, we shall also establish a $q$-analogue of \eqref{eq:div-3-2}.
\begin{theorem}\label{thm:2}
Let $n$ be a positive odd integer. Then
\begin{align}
\sum_{k=0}^{n-1} (-1)^k [3k+1]\frac{(q;q^2)_k^3}{(q;q)_k^3}
&\equiv
 (-1)^{(n-1)/2}q^{(1-n^2)/4}
\left(q^{n\choose 2}[n]+\frac{(n^2-1)(1-q)^2}{24}[n]^3\right) \notag \\
&\quad+[n]^3\sum_{k=1}^{(n-1)/2}\frac{q^k (q^2;q^2)_{k}}{[2k][2k-1](q;q^2)_{k}} \pmod{[n]\Phi_n(q)^3}. \label{eq:main-2}
\end{align}
\end{theorem}


The remainder of the paper proceeds as follows. We prove Theorems \ref{thm:1} and \ref{thm:2}
in Sections 2 and 3, respectively, by using the $q$-WZ method, together with  a $q$-analogue of Wolstenholme's congruence and a $q$-analogue of Morley's congruence.
In Section 4, we give some concluding remarks and two open
problems. Particularly in Corollary \ref{cor:one}, using a recent result of Wei \cite{Wei},  we shall deduce a $q$-analogue of another supercongruence related to
Euler numbers of Sun from Theorem \ref{thm:1}.

\section{Proof of Theorem \ref{thm:1}}
Recall that the {\it $q$-binomial coefficients} ${M\brack N}$
are defined by
$$
{M\brack N}={M\brack N}_q
=\begin{cases}\displaystyle\frac{(q;q)_M}{(q;q)_N(q;q)_{M-N}} &\text{if $0\leqslant N\leqslant M$,} \\[10pt]
0 &\text{otherwise.}
\end{cases}
$$
We need the following $q$-analogue of Wolstenholme's congruence (see \cite[Lemma 3.1]{GW}).
\begin{lemma}\label{lem:first}
Let $n$ be a positive integer. Then
\begin{align*}
{2n-1\brack n-1} \equiv (-1)^{n-1}q^{n\choose 2}+\frac{(n^2-1)(1-q)^2}{12}[n]^2
 \pmod{\Phi_n(q)^3}.  
\end{align*}
\end{lemma}

Moreover, a $q$-analogue of Morley's congruence (see \cite[(1.5)]{LPZ})
and a $q$-analogue of Fermat's little theorem (see \cite[Lemma 3.2]{GW}) will also be used in our proof.
\begin{lemma}\label{lem:second}
Let $n$ be a positive odd integer. Then, modulo $\Phi_n(q)^3$,
\begin{align*}
{n-1\brack \frac{n-1}{2}}_{q^2}
&\equiv (-1)^{(n-1)/2}q^{(1-n^2)/4}\left((-q;q)_{n-1}^2-\frac{(n^2-1)(1-q)^2}{24}[n]^2\right).
\end{align*}
\end{lemma}

\begin{lemma}\label{lem:third}
Let $n$ be a positive odd integer. Then
\begin{align}
(-q;q)_{n-1}\equiv 1\pmod{\Phi_n(q)}.  \label{eq:phi}
\end{align}
\end{lemma}

\begin{proof}[Proof of Theorem {\rm\ref{thm:1}}]
By \cite[Theorem 6.1]{Guo-mod4}, modulo $[n]\Phi_n(q)(1-aq^n)(a-q^n)$,
\begin{align}
&\sum_{k=0}^{(n-1)/2}[4k+1]\frac{(aq;q^2)_k (q/a;q^2)_k (q/b;q^2)_k (q;q^2)_k}
{(aq^2;q^2)_k(q^2/a;q^2)_k (bq^2;q^2)_k (q^2;q^2)_k}b^k \notag\\[5pt]
&\quad\equiv \sum_{k=0}^{n-1}[4k+1]\frac{(aq;q^2)_k (q/a;q^2)_k (q/b;q^2)_k (q;q^2)_k}
{(aq^2;q^2)_k(q^2/a;q^2)_k (bq^2;q^2)_k (q^2;q^2)_k}b^k, \label{eq:equiv}
\end{align}
where $a$ and $b$ are indeterminates.
Letting $b\to\infty$ and $a=1$ in \eqref{eq:equiv}, we get
\begin{align}
\sum_{k=0}^{(n-1)/2}(-1)^k q^{k^2} [4k+1]\frac{(q;q^2)_k^3}{(q^2;q^2)_k^3}
\equiv \sum_{k=0}^{n-1}(-1)^k q^{k^2} [4k+1]\frac{(q;q^2)_k^3}{(q^2;q^2)_k^3} \pmod{\Phi_n(q)^4}.  \label{eq:equiv-2}
\end{align}
By \cite[Theorem 4.1]{GuoZu}, both sides of \eqref{eq:equiv-2} are congruent to $0$ modulo $[n]$, and so
it is also true modulo $[n]\Phi_n(q)^3$. Thus, to prove Theorem \ref{thm:1}, it suffices to prove
the $N=(n-1)/2$ case.

We introduce two rational functions in $q$:
\begin{align*}
F(m,k) &=(-1)^{m+k}q^{(m-k)^2}\frac{[4m+1](q;q^2)_{m}^2(q;q^2)_{m+k}}{(q^2;q^2)_{m}^2(q^2;q^2)_{m-k}(q;q^2)_{k}^2}, \\[5pt]
G(m,k) &=\frac{(-1)^{m+k}q^{(m-k)^2}(q;q^2)_{m}^2(q;q^2)_{m+k-1}}{(1-q)(q^2;q^2)_{m-1}^2(q^2;q^2)_{m-k}(q;q^2)_{k}^2},
\end{align*}
where we assume that $1/(q^2;q^2)_{M}=0$ for negative integers $M$. As mentioned in \cite{Guo2018}, the functions $F(m,k)$ and $G(m,k)$ form a $q$-WZ pair.
Namely, they satisfy the following equality
\begin{align}
F(m,k-1)-F(m,k)=G(m+1,k)-G(m,k).  \label{eq:fnk-gnk}
\end{align}
Summing \eqref{eq:fnk-gnk} over $m$ from $0$ to $(n-1)/2$, we get
\begin{align}
\sum_{m=0}^{(n-1)/2}F(m,k-1)-\sum_{m=0}^{(n-1)/2}F(m,k)=G\left(\frac{n+1}{2},k\right)-G(0,k)=G\left(\frac{n+1}{2},k\right).  \label{eq:fnk-gn0-00}
\end{align}
Summing \eqref{eq:fnk-gn0-00} further over $k$ from $1$ to $(n-1)/2$ and noticing that $F(m,(n-1)/2)=0$ for $m<(n-1)/2$, we obtain
\begin{align}
\sum_{m=0}^{(n-1)/2}F(m,0)-F\left(\frac{n-1}{2},\frac{n-1}{2}\right)=\sum_{k=1}^{(n-1)/2}G\left(\frac{n+1}{2},k\right).  \label{eq:fnk-gn-new}
\end{align}

Note that, for $k=1,2,\ldots,(n-1)/2$, we have
\begin{align}
G\left(\frac{n+1}{2},k\right)
&=\frac{(-1)^{(n+1)/2+k}q^{((n+1)/2-k)^2}(q;q^2)_{(n+1)/2}^2(q;q^2)_{(n-1)/2+k}}{(1-q)(q^2;q^2)_{(n-1)/2}^2(q^2;q^2)_{(n+1)/2-k}(q;q^2)_{k}^2} \notag \\[5pt]
&=(-1)^{(n+1)/2+k}q^{((n+1)/2-k)^2}\frac{(1-q^n)^3(q;q^2)_{(n-1)/2}^3 (q^{n+2};q^2)_{k-1}}{(1-q)(q^2;q^2)_{(n-1)/2}^2(q^2;q^2)_{(n+1)/2-k}(q;q^2)_{k}^2}.
\label{eq:gppk}
\end{align}
Since $q^n\equiv 1\pmod{\Phi_n(q)}$, there hold
\begin{align*}
(q^2;q^2)_{(n+1)/2-k}&=\frac{(q^2;q^2)_{(n-1)/2}}{(q^{n+3-2k};q^2)_{k-1}} \\[5pt]
&\equiv \frac{(q^2;q^2)_{(n-1)/2}}{(q^{3-2k};q^2)_{k-1}}
=(-1)^{k-1}q^{(k-1)^2}\frac{(q^2;q^2)_{(n-1)/2}}{(q;q^2)_{k-1}} \pmod{\Phi_n(q)},
\end{align*}
and
\begin{align*}
\frac{(q;q^2)_{(n-1)/2}}{(q^2;q^2)_{(n-1)/2}}&=\prod_{j=1}^{(n-1)/2}\frac{1-q^{2j-1}}{1-q^{n-2j+1}} \\[5pt]
&\equiv \prod_{j=1}^{(n-1)/2}\frac{1-q^{2j-1}}{1-q^{1-2j}}=(-1)^{(n-1)/2}q^{(n-1)^2/4} \pmod{\Phi_n(q)}.
\end{align*}
Employing the above two $q$-congruences, we deduce from \eqref{eq:gppk} that, for $1\leqslant k\leqslant (n-1)/2$,
\begin{align}
G\left(\frac{n+1}{2},k\right)
\equiv \frac{q^k(1-q^n)^3 (q^{2};q^2)_{k-1}}{(1-q)(1-q^{2k-1})(q;q^2)_{k}}
=\frac{q^k[n]^3 (q^{2};q^2)_{k}}{[2k][2k-1](q;q^2)_{k}} \pmod{\Phi_n(q)^4}.
\label{eq:gppk-new}
\end{align}
Since $(q;q^2)_k/(q^2;q^2)_k={2k\brack k}/(-q;q)_k^2$ and the $q$-binomial coefficient can be written as a product of different cyclotomic polynomials
(see \cite[Lemma 1]{CH}), we see that the right-hand side of \eqref{eq:gppk-new} is congruent to $0$ modulo $[n]$, and so is \eqref{eq:gppk}.
Namely, the $q$-congruence \eqref{eq:gppk-new} holds modulo $[n]\Phi_n(q)^3$.

In addition, by Lemmas \ref{lem:first}--\ref{lem:third},
\begin{align}
&F\left(\frac{n-1}{2},\frac{n-1}{2}\right) \notag\\[5pt]
&\quad=\frac{[n]}{(-q;q)_{n-1}^2}{2n-1\brack n-1}{n-1\brack \frac{n-1}{2}}_{q^2} \notag \\[5pt]
&\quad\equiv (-1)^{(n-1)/2}q^{(1-n^2)/4}
\left\{q^{n\choose 2}[n]+\frac{(n^2-1)(1-q)^2}{24}\left(2-\frac{q^{n\choose 2}}{(-q;q)_{n-1}^2}\right)[n]^3\right\} \notag\\[5pt]
&\quad\equiv (-1)^{(n-1)/2}q^{(1-n^2)/4}
\left(q^{n\choose 2}[n]+\frac{(n^2-1)(1-q)^2}{24}[n]^3\right) \pmod{[n]\Phi_n(q)^3}, \label{eq:abcd}
\end{align}
where we have used $q^{n\choose 2}\equiv 1\pmod{\Phi_n(q)}$ for odd $n$ in the last step.

Substituting \eqref{eq:abcd} and the modulus $[n]\Phi_n(q)^3$ case of \eqref{eq:gppk-new} into \eqref{eq:fnk-gn-new},
we are led to \eqref{eq:main-1} in the case where $N$ is equal to $(n-1)/2$. This completes the proof of the theorem.
\end{proof}

\section{Proof of Theorem \ref{thm:2}}
The author \cite{Guo-div} employed the following functions
\begin{align*}
F(m,k) &=(-1)^{m}[3m-2k+1]{2m-2k\brack m}\frac{(q;q^2)_{m}(q;q^2)_{m-k} }{(q;q)_{m} (q^2;q^2)_{m-k}}, \\[5pt]
G(m,k) &=(-1)^{m+1}[m]{2m-2k\brack m-1}\frac{(q;q^2)_{m}(q;q^2)_{m-k} q^{m+1-2k} }{(q;q)_{m} (q^2;q^2)_{m-k}}
\end{align*}
to establish \eqref{eq:q-Zudilin-3}.
It is not difficult to verify that $F(m,k)$ and $G(m,k)$ satisfy the relation
\begin{align}
F(m,k-1)-F(m,k)=G(m+1,k)-G(m,k).  \label{eq:fnk-gnk-new}
\end{align}
That is, they form a $q$-WZ pair.

Since \eqref{eq:main-2} is clearly true for $n=1$, we now assume that $n\geqslant 3$.
Summing \eqref{eq:fnk-gnk-new} over $m=0,1,\ldots,n-1$, we obtain
\begin{align}
\sum_{m=0}^{n-1}F(m,k-1)-\sum_{m=0}^{n-1}F(m,k)=G(n,k).  \label{eq:fnk-gn0-0011}
\end{align}
Summing \eqref{eq:fnk-gn0-0011} further over $k=1,\ldots,n-1$ and noticing that $F(m,n-1)=0$ for $m\leqslant n-1$, we arrive at
\begin{align}
\sum_{m=0}^{n-1}F(m,0)=\sum_{k=1}^{n-1}G(n,k)=\sum_{k=1}^{(n+1)/2}G(n,k).  \label{eq:fm0-gmk}
\end{align}
In view of
$$
{2m\brack m}
=\frac{(q;q^2)_m (-q;q)_m^2}{(q^2;q^2)_m},
$$
the identity \eqref{eq:fm0-gmk} can be written as
\begin{align}
\sum_{m=0}^{n-1}(-1)^n [3m+1]\frac{(q;q^2)_m^3}{(q;q)_m^3}
=\frac{[n]{2n-1\brack n-1}}{(-q;q)_{n-1}}\sum_{k=1}^{(n+1)/2}{2n-2k\brack n-1}\frac{(q;q^2)_{n-k} q^{n+1-2k} }{(q^2;q^2)_{n-k}}.
\label{eq:fn03n+1}
\end{align}

For $1\leqslant k\leqslant (n-1)/2$, we have
\begin{align*}
\frac{(q;q^2)_{n-k}}{(q^2;q^2)_{n-k}}
&=\frac{(1-q^n)(q;q^2)_{(n-1)/2}(q^{n+2};q^2)_{(n-1)/2-k}}{(1-q^{n-1})(q^2;q^2)_{(n-3)/2}(q^{n+1};q^2)_{(n+1)/2-k}} \\[5pt]
&\equiv \frac{(1-q^n)(q;q^2)_{(n-1)/2}(q^2;q^2)_{(n-1)/2-k}}{(1-q^{n-1})(q^2;q^2)_{(n-3)/2}(q;q^2)_{(n+1)/2-k}}  \\
&=\frac{(1-q^n)(q^{n+2-2k};q^2)_{k-1}}{(1-q^{n-1})(q^{n+1-2k};q^2)_{k-1}} \\
&\equiv \frac{(1-q^n)(q^{2-2k};q^2)_{k-1}}{(1-q^{-1})(q^{1-2k};q^2)_{k-1}}
=-\frac{q^k(1-q^n)(q^2;q^2)_{k-1}}{(q;q^2)_{k}} \pmod{\Phi_n(q)^2},
\end{align*}
and
\begin{align*}
{2n-2k+1\brack n}=\prod_{k=1}^{n-2k+1}\frac{1-q^{n+j}}{1-q^{j}}\equiv 1\pmod{\Phi_n(q)}.
\end{align*}
It follows that
\begin{align}
&\sum_{k=1}^{(n+1)/2}{2n-2k\brack n-1}\frac{(q;q^2)_{n-k} q^{n+1-2k} }{(q^2;q^2)_{n-k}} \notag\\[5pt]
&\quad=\frac{(q;q^2)_{(n-1)/2}}{(q^2;q^2)_{(n-1)/2}}+\sum_{k=1}^{(n-1)/2}{2n-2k\brack n-1}\frac{(q;q^2)_{n-k} q^{n+1-2k} }{(q^2;q^2)_{n-k}} \notag\\[5pt]
&\quad={n-1\brack \frac{n-1}{2}}_{q^2}\frac{1}{(-q;q)_{n-1}}
+\sum_{k=1}^{(n-1)/2}{2n-2k+1\brack n}\frac{(1-q^n)(q;q^2)_{n-k} q^{n+1-2k} }{(1-q^{2n-2k+1})(q^2;q^2)_{n-k}} \notag\\[5pt]
&\quad\equiv (-1)^{(n-1)/2}q^{(1-n^2)/4}\left((-q;q)_{n-1}-\frac{(n^2-1)(1-q)^2}{24}[n]^2\right)  \notag\\[5pt]
&\quad\quad+[n]^2\sum_{k=1}^{(n-1)/2}\frac{q^k  (q^2;q^2)_k}{[2k][2k-1] (q;q^2)_k} \pmod{\Phi_n(q)^3}, \label{eq:2nk}
\end{align}
where we have used Lemmas \ref{lem:second} and \ref{lem:third} in the last step.

By Lemmas \ref{lem:first} and \ref{lem:third}, we have
\begin{align}
\frac{{2n-1\brack n-1}}{(-q;q)_{n-1}} \equiv \frac{(-1)^{n-1}q^{n\choose 2}}{(-q;q)_{n-1}}+\frac{(n^2-1)(1-q)^2}{12}[n]^2
 \pmod{\Phi_n(q)^3}.   \label{eq:2nk-2}
\end{align}
Substituting \eqref{eq:2nk} and \eqref{eq:2nk-2}  into \eqref{eq:fn03n+1} and making some simplifications, we immediately obtain
\eqref{eq:main-2}.

\section{Concluding remarks and open problems}
From \eqref{eq:Sun} and \eqref{eq:div-3-2} one sees that, for any prime $p>3$,
\begin{align}
\sum^{m}_{k=0}(-1)^k(4k+1)\frac{(\frac{1}{2})_k^3}{k!^3}
\equiv \sum_{k=0}^{p-1}(-1)^k (3k+1)\frac{(\frac{1}{2})_k^3}{k!^3} 2^{3k}\pmod{p^4}, \label{eq:sun-combin}
\end{align}
where $m=p-1$ or $(p-1)/2$. Further, combining \eqref{eq:main-1} and \eqref{eq:main-2}, we have the following  $q$-analogue of \eqref{eq:sun-combin}:
for any positive odd integer $n$,
\begin{equation}
\sum_{k=0}^{N}
(-1)^k q^{k^2} [4k+1]\frac{(q;q^2)_k^3}{(q^2;q^2)_k^3}
\equiv
\sum_{k=0}^{n-1}
(-1)^k [3k+1]\frac{(q;q^2)_k^3}{(q;q)_k^3} \pmod{[n]\Phi_n(q)^3},  \label{eq:q4k-3k}
\end{equation}
where $N=(n-1)/2$ or $n-1$.

We now provide a conjectural parametric generalization of \eqref{eq:q4k-3k}.
\begin{conjecture}Let $n$ be a positive odd integer. Then, modulo $[n]\Phi_n(q)(1-aq^n)(a-q^n)$,
\begin{align}
&\sum_{k=0}^{N}
(-1)^k q^{k^2} [4k+1]\frac{(aq;q^2)_k (q/a;q^2)_k (q;q^2)_k }{(aq^2;q^2)_k (q^2/a;q^2)_k (q^2;q^2)_k } \notag\\[5pt]
&\quad\equiv
\sum_{k=0}^{n-1}
(-1)^k [3k+1]\frac{(aq;q^2)_k (q/a;q^2)_k (q;q^2)_k }{(aq;q)_k (q/a;q)_k (q;q)_k },  \label{eq:q4k-3k-a}
\end{align}
where $N=(n-1)/2$ or $n-1$.
\end{conjecture}

Using the `creative microscoping' method introduced in \cite{GuoZu} and the Chinese remainder theorem for relatively prime polynomials,
the author \cite[Theorem 5.3]{Guo-mod4} has shown that the left-hand side of \eqref{eq:q4k-3k-a}
is congruent to
\begin{align*}
&(-1)^{(n-1)/2}q^{(n-1)^2/4}[n]+(-1)^{(n-1)/2}q^{(n-1)^2/4}\frac{(1-aq^n)(a-q^n)}{(1-a)^2}[n]  \\[5pt]
&\quad{}-\frac{(1-aq^n)(a-q^n)}{(1-a)^2} [n]\sum_{k=0}^{(n-1)/2}\frac{(q;q^2)_k^2}{(aq^2;q^2)_k(q^2/a;q^2)_k}
\end{align*}
modulo $[n]\Phi_n(q)(1-aq^n)(a-q^n)$. But it seems rather difficult to prove that the right-hand side of
\eqref{eq:q4k-3k-a} is also congruent to the above expression, though a three-parametric generalization of
\eqref{eq:q-Zudilin-3} was already proved by the author and Schlosser \cite[Theorem 6.1]{GS} (see also \cite[Conjecture 4.6]{GuoZu}).

In 2011, Sun \cite{Sun0} studied many interesting supercongruences related to Euler numbers. In particular, Sun \cite[Theorems 1.1 and 1.2]{Sun0} proved that, for any prime $p>3$,
\begin{align}
\sum_{k=0}^{p-1}\frac{1}{2^k}{2k\choose k}
&\equiv (-1)^{(p-1)/2}-p^2 E_{p-3} \pmod{p^3}, \label{eq:Sun-2}\\[5pt]
\sum_{k=0}^{(p-1)/2}\frac{1}{16^k}{2k\choose k}^2
&\equiv (-1)^{(p-1)/2}+p^2 E_{p-3} \pmod{p^3}, \label{eq:Sun-3}\\[5pt]
\sum_{k=0}^{p-1}\frac{1}{16^k}{2k\choose k}^2
&\equiv (-1)^{(p-1)/2}-p^2 E_{p-3} \pmod{p^3}.  \label{eq:Sun-4}
\end{align}
Recently, Wei \cite[Theorem 1.1 with $c=q$, $d\to\infty$]{Wei} gave the following result: for any positive odd integer $n$,
and $N=(n-1)/2$ or $n-1$,
\begin{align}
\sum_{k=0}^{N}
(-1)^k q^{k^2} [4k+1]\frac{(q;q^2)_k^3}{(q^2;q^2)_k^3}
&\equiv
q^{(1-n)/2}\left([n]+\frac{(n^2-1)(1-q)^2}{24}[n]^3\right) \notag\\[5pt]
&\quad\times\sum_{k=0}^{(n-1)/2}
\frac{(q;q^2)_k^2}{(q^2;q^2)_k^2} q^{2k}  \pmod{[n]\Phi_n(q)^3},  \label{eq:q4k-3k}
\end{align}
which is clearly a $q$-analogue of the relation between \eqref{eq:Sun} and \eqref{eq:Sun-3}: for any prime $p>3$,
\begin{align*}
\sum^{m}_{k=0}(-1)^k(4k+1)\frac{(\frac{1}{2})_k^3}{k!^3}
\equiv p\sum_{k=0}^{(p-1)/2}\frac{1}{16^k}{2k\choose k}^2\pmod{p^4}.
\end{align*}
where $m=p-1$ or $(p-1)/2$.

Note that the author, Pan and Zhang \cite{GPZ} gave the following $q$-supercongruence:
\begin{align}
\sum_{k=0}^{(n-1)/2}\frac{(q;q^2)_k^2}{(q^2;q^2)_k^2}q^{2k}\equiv (-1)^{(n-1)/2}q^{(n^2-1)/4} \pmod{\Phi_n(q)^2}  \label{eq:guozeng}
\end{align}
(see \cite{GZ14} for $n$ being a prime). Hence, the $q$-supercongruence \eqref{eq:q4k-3k} may be written as
\begin{align}
& q^{(n-1)/2}\sum_{k=0}^{N}
(-1)^k q^{k^2} [4k+1]\frac{(q;q^2)_k^3}{(q^2;q^2)_k^3} \notag\\[5pt]
&\equiv
[n]\sum_{k=0}^{(n-1)/2}
\frac{(q;q^2)_k^2}{(q^2;q^2)_k^2} q^{2k}
+\frac{(n^2-1)(1-q)^2}{24}[n]^3 (-1)^{(n-1)/2}q^{(n^2-1)/4} \pmod{[n]\Phi_n(q)^3},  \label{eq:q4k-3k-2}
\end{align}
Combining \eqref{eq:main-1} and \eqref{eq:q4k-3k-2} we immediately obtain a $q$-analogue of \eqref{eq:Sun-3}.
\begin{corollary}  \label{cor:one}
Let $n$ be a positive odd integer. Then
\begin{align*}
\sum_{k=0}^{(n-1)/2}
\frac{(q;q^2)_k^2}{(q^2;q^2)_k^2} q^{2k}
&\equiv
(-1)^{(n-1)/2}q^{(n^2-1)/4}  \\[5pt]
&\quad+(-1)^{(n-1)/2}\frac{(n^2-1)(1-q)^2}{24}[n]^2 \left(q^{-(n-1)^2/4}-q^{(n^2-1)/4}\right) \\[5pt]
&\quad+ q^{(n-1)/2}[n]^2\sum_{k=1}^{(n-1)/2}\frac{q^k (q^2;q^2)_{k}}{[2k][2k-1](q;q^2)_{k}} \pmod{\Phi_n(q)^3}.
\end{align*}
\end{corollary}

However, to the best of the author's knowledge,
no $q$-analogues of \eqref{eq:Sun-2} and \eqref{eq:Sun-4}, even conjectural, are known in the literature.
It follows from \eqref{eq:Sun-2} and \eqref{eq:Sun-4} that, for any odd prime $p$,
\begin{equation}
\sum_{k=0}^{p-1}\frac{1}{2^k}{2k\choose k}
\equiv
\sum_{k=0}^{p-1}\frac{1}{16^k}{2k\choose k}^2  \pmod{p^3}.  \label{eq:Sun2-4}
\end{equation}
On the other hand, for odd $n$, the author \cite{Guo-t} proved that
\begin{align}
\sum_{k=0}^{n-1}\frac{q^k}{(-q;q)_{k}}{2k\brack k}\equiv (-1)^{(n-1)/2}q^{(n^2-1)/4} \pmod{\Phi_n(q)^2}, \label{eq:q-tauraso}
\end{align}
confirming a conjecture of Tauraso \cite{Tauraso}.

In light of \eqref{eq:guozeng} and \eqref{eq:q-tauraso}, we believe that the following $q$-analogue of \eqref{eq:Sun2-4} is true.
\begin{conjecture}
Let $n$ be a positive odd integer. Then
\begin{align*}
\sum_{k=0}^{n-1}\frac{q^k}{(-q;q)_k}{2k\brack k} \equiv
\sum_{k=0}^{n-1}\frac{(q;q^2)_k^2}{(q^2;q^2)_k^2}q^{2k}
\pmod{\Phi_n(q)^3}.
\end{align*}
\end{conjecture}

\end{document}